\documentclass[12pt]{amsart}
\usepackage[margin=1in]{geometry}

\RequirePackage{tikz}
\usetikzlibrary{decorations.fractals, arrows}

 \newtheorem{theorem}{Theorem}[section]
 \newtheorem{definition}[theorem]{Definition}
 \newtheorem{proposition}[theorem]{Proposition}
 \newtheorem{lemma}[theorem]{Lemma}
 
 \newtheorem{remark}[theorem]{Remark}
 \newtheorem{example}[theorem]{Example}
 \newtheorem{conjecture}[theorem]{Conjecture}

\title{On a topological Erd\H{o}s similarity problem}
\author{John Gallagher}
\address{San Francisco State University, Department of Mathematics, 1600 Holloway Avenue, CA 94132, US}
\email{jgallagher1@mail.sfsu.edu}
\author{Chun-Kit Lai}
\address{ San Francisco State University, Department of Mathematics, 1600 Holloway Avenue, CA 94132, US}
\email{cklai@sfsu.edu}
\author{Eric Weber}
\address{ Iowa State University, Department of Mathematics, 396 Carver Hall, 411 Morrill Rd., Ames, IA, 50011}
\email{esweber@iastate.edu}
\date{\today}

\newcommand{\Z}{\mathbb{Z}}
\newcommand{\Q}{\mathbb{Q}}
\newcommand{\R}{\mathbb{R}}


\begin{document}
\maketitle

\begin{abstract}
     A pattern is called universal in another collection of sets, when every set in the collection contains some linear and translated copy of the original pattern.  Paul Erd\H{o}s proposed a conjecture that no infinite set is universal in the  collection of sets with positive measure.  This paper explores an analogous problem in the topological setting. Instead of sets with positive measure, we investigate the collection of dense $G_{\delta}$ sets and in the collection of generic sets (dense $G_{\delta}$ and complement has Lebesgue measure zero). We refer to such pattern as topologically universal and generically universal respectively.  It  is  easy to show that any countable set is topologically universal, while any set containing an interior cannot be topologically universal. In this paper, we will show that Cantor sets on $\R^d$ are not topologically universal and Cantor sets  with positive Newhouse thickness on $\R^1$ are not generically universal.  This gives a positive partial answer to a  question by Svetic concerning the Erd\H{o}s similarity problem on Cantor sets. Moreover, we also obtain a higher dimensional generalization of the generic universality problem. 

\end{abstract}

\section{Introduction}

%An informal problem that appears in data science or even in daily life is that given a geometric pattern (which may represent some customer behavior), finite or infinite, can we find a copy of the pattern in a collection of data sets? One can actually formulate this problem mathematically. Our focus will be on the Euclidean space ${\mathbb R}^d$, though it is possible to formulate it even more abstractly.  We first need to specify what kind of copy we need to find and what kind of ``data set" we need to work on. 

A question that frequently arises has the following generic form: Does every ``large'' (or unstructured) set possess a ``copy'' of a ``small'' (or structured) set?  For example, Erd\H{o}s and Tur\'an conjectured that every $X\subset {\mathbb N}$ of positive density (large, unstructured) contains a copy of the set $\{1,2,...,n\}$ (small, structured) in the form of an arithmetic progression. The conjecture was famously proven true by Szemer\'edi \cite{Sze75,AC_Tao}.  In a similar vein, Steinhaus proved that the difference set of a set of positive measure in $\mathbb{R}$ (large) contains a scaled copy of the interval $(-1,1)$ (small), \cite{Steinhaus}.

The words ``large'', ``small'', and ``copy'' can take on multiple forms, so we begin by defining some of our terms.

\begin{definition}
Let $E\subset{\mathbb R}^d$ be a set and let ${\mathcal X}$ be a collection of subsets in ${\mathbb R}^d$. 
\begin{enumerate}
\item  An {\bf affine copy} of  $E$ is a copy of the form $t+ T(E)$ where $t\in{\mathbb R}^d$ and $T$ is an invertible linear transformation on ${\mathbb R}^d$. A {\bf similar copy} of $E$ is an affine copy such that $T = \lambda O$ where $\lambda>0$ is a scalar and $O$ is an orthogonal transformation. 

\medskip

\item We say that $E$ is {\bf universal in ${\mathcal X}$} if for every $K\in {\mathcal X}$, there exists an affine copy of $E$, $t+T(E)$, such that $t+T(E)\subset K$. 

\medskip

\item We say that $E$ is {\bf measure-universal} if $E$ is universal in ${\mathcal X}$, where ${\mathcal X}$ is taken to be the collection of all Lebesgue measurable set with positive Lebesgue measure. 
\end{enumerate}
\end{definition}

\medskip

In one dimension, affine copies and similar copies coincide and they are of the form $t+\lambda E$ where $t\in{\mathbb R}$ and $\lambda\ne 0$.  Many problems in mathematics can be formulated in terms of universality. 
%A famous theorem of Szemer\'{e}di states that every set $X\subset {\mathbb N}$ of positive density contains arbitrarily long arithmetic progression (see e.g. \cite{AC_Tao}).  
Szemer\'edi's theorem then can be stated as: the set $\{1,2,...,n\}$ is universal in the collection of sets of positive density in ${\mathbb N}$. The Toeplitz square peg problem asserted that every Jordan curve admits four points on the curve forming a square. Formulated in our notation and interpreting universality in terms of similarity copy, it means that the unit square corners $\{(0,0), (1,0), (0,1), (1,1)\}$ is universal in the collection of all Jordan curves. The problem was recently solved for smooth Jordan curves \cite{Square-Peg}. 

\medskip

Our notion of universality was first introduced by Kolountzakis \cite{Kolo}, in which the goal was to study the famous Erd\H{o}s similarity conjecture. 

\medskip

{\bf Conjecture (Erd\H{o}s):} There is no set of infinite cardinality that is measure-universal.

\medskip

Steinhaus \cite{Steinhaus} first showed that finite sets are measure-universal. This motivated Paul Erd\H{o}s to pose the conjecture back in 1974 and he offered \$100 for solving this problem. The conjecture is still open; for a survey of the problem, one can refer to \cite{Svetic}. Let us summarize some progress here. With a simple observation, we can see that the conjecture can be resolved in its full generality if we can show that all positive decreasing sequences whose limit is zero are not measure-universal.   Falconer \cite{Falconer} made a substantial progress by showing that slowly decaying sequences $\{x_n: n\in{\mathbb N}\}$ in the sense that 
$$
\liminf_{n\to\infty} \frac{x_{n}}{x_{n+1}}=1
$$
are not measure-universal.  Bourgain \cite{Bourgain} demonstrated that the sum-set $S_1+S_2+S_3$ of any three infinite sets $S_1,S_2,S_3$ cannot be measure-universal. Kolountzakis \cite{Kolo} demonstrated using probabilistic arguments that certain set with large gaps cannot be measure-universal.  Currently it is still an open question whether or not exponential decaying sequences such as $\{2^{-n}\}$ are measure-universal. Cruz, the second-named author and Pramanik recently constructed a Cantor set $K$ such that the set of Erd\H{o}s points in $K$, i.e.
$$
\{x\in K: \forall\delta\ne 0, \ x+\delta \{2^{-n}\}\not\subset K\},
$$
has Hausdorff dimension 1.  If one could show that the above set could be of positive Lebesgue measure, the Erd\H{o}s similarity problem will be solved for $\{2^{-n}\}$. Their result also works on sequences which do not reach super-exponential decay \cite{CLP2022}.

\subsection{Main Results.} The main purpose of this paper is to study a topological version of the Erd\H{o}s similarity problem. If we regard a set of positive Lebesgue measure as measure-theoretically large, then a dense $G_{\delta}$ set will be regarded as topologically large. Recall that a $G_{\delta}$ set is a set $G$ that can be written as countable intersection of open sets. If each open set is dense in $\R^d$, then the well-known Baire Category theorem shows that $G$ is a dense and uncountable set.  There is no relation between sets with positive Lebesgue measure and  dense $G_{\delta}$ sets. A fat Cantor set has positive Lebesgue measure, but is nowhere dense. On the other hand, the set of all Liouville's numbers is a dense $G_{\delta}$ but with Lebesgue measure zero (Hausdorff dimension zero indeed).

\medskip

\begin{definition}
\begin{enumerate}
    \item We say that a set $E\subset {\mathbb R}^d$ is {\bf topologically universal } if $E$ is universal in the collection of all dense $G_{\delta}$ sets in ${\mathbb R}^d$.
    
    \medskip
    
    \item We say that a set $E\subset {\mathbb R}^d$ is {\bf generically universal } if $E$ is universal in the collection of all dense $G_{\delta}$ sets $G$ such that $m({\mathbb R}^d\setminus G) = 0$ in ${\mathbb R}^d$ (Here $m$ denotes the Lebesgue measure).
\end{enumerate}

\end{definition}

In the first definition, we are interested in what set is universal for topologically large sets. In the second definition, we notice that $m(G) = \infty$, so we are interested in what set is universal for both measure-theoretically and  topologically large sets (such sets are sometimes referred to as {\it generic sets}, which is the reason why we choose this definition). 

\medskip

It is a simple observation from the Baire Category theorem that all countable sets are topologically universal. On the other hand, a set containing an interior point cannot be topologically or generically universal because there are dense $G_{\delta}$ sets with full Lebesgue measures with empty interior. As any affine copy of a set with interior must have interior, a dense $G_{\delta}$ set with empty interior  cannot contain any such affine copy. Hence, our focus will be on whether nowhere dense sets are topologically universal. Let us first make precise the meaning of Cantor set in our setting. 

\medskip
\begin{definition}
$E$ is a Cantor set in ${\mathbb R}^d$ if it is a totally disconnected, perfect and compact subset of ${\mathbb R}^d$.  
\end{definition}

\medskip

Because of the existence of $G_{\delta}$ sets that have Hausdorff dimension zero, by monotonicity of Hausdorff measures, all sets of positive Hausdorff dimension cannot be contained inside such a $G_{\delta}$ set, and hence sets of positive Hausdorff dimension are topologically non-universal. Our first theorem is to show that by considering arbitrary dimension functions, no Cantor sets are topologically universal.

\medskip

\begin{theorem}\label{theorem_invisible}
    For any Cantor set $E\subset \R^d$, there exists a dense $G_{\delta}$ set such that it does not contain an affine copy of $E$. Consequently, there do not exist any topologically universal Cantor sets on $\R^d$.
\end{theorem}

We now turn to study  generic  universality of Cantor sets. Because $m(\R^d\setminus G) =0$ for any generic $G_{\delta}$ set, $m(G)$ must be infinity and thus those arbitrarily small dimensional dense $G_{\delta}$ sets no longer exist. Generic universality is also a closer analogue to the Erd\H{o}s similarity conjecture because we now require the sets of interest having positive Lebesgue measure. 

\medskip

We first focus on Cantor sets in ${\mathbb R}^1$.  In addition to Hausdorff dimension, Newhouse thickness of a Cantor set (see Section \ref{sec:Newhouse} for the precise definition) has been another useful quantity to describe the size of Cantor sets.  In particular, the gap lemma provides a natural sufficient condition for two thick Cantor sets to intersect. Our main theorem is the following:

\medskip
\begin{theorem}\label{theorem_positive_NW}
 There exists a dense $G_{\delta}$ set $G$ with $m(\R\setminus G) =0$ such that for all Cantor sets $J$ with positive Newhouse thickness, $G$ does not contain an affine copy of $J$. 

\medskip

As a consequence, Cantor sets on ${\mathbb R}^1$ with positive Newhouse thickness are not generically universal. 
\end{theorem}

Our theorem also tells us something about the measure non-universality of Cantor sets. Although the Erd\H{o}s similarity conjecture can be resolved if we can show that all decreasing sequence are not universal, it is not even an easy question to show that a Cantor set is measure non-universal. Indeed,  Svetic \cite{Svetic}  proposed the following stronger question in this regard. {\it ``Is it true that for every uncountably infinite set, $E$, of real numbers, there exists $S \subset [0,1]$ of full measure that does not contain an affine copy of $E$?''} Notice that if a set is generically non-universal, then it must be measure non-universal. 

\medskip

Our Theorem \ref{theorem_positive_NW} now answers Svetic's question in a very strong way.  There exists a fixed set, namely $S = G \cap [0,1]$ where $G$ is defined in Theorem \ref{theorem_positive_NW}, of full Lebesgue measure in $[0,1]$ which doesn't contain  affine copies of any Cantor sets with positive thickness. 
%Our Theorem \ref{theorem_positive_NW} gives a positive answer to Svetic's question for Cantor sets with positive Newhouse thickness.
%\begin{corollary}\label{theorem_Svetic_question}
 %   Let $E$ be a Cantor set with positive Newhouse thickness.  Then there exists a set of full Lebesgue measure in $[0,1]$ that does not contain an affine copy of $E$. 
%\end{corollary}

\medskip

We now consider higher dimensions. First, one can show that a set containing a path-connected component cannot be generically universal (see Proposition \ref{prop_path-connected}). Therefore, our main interest will be focused on totally disconnected Cantor set. There has been recent work on generalizing the gap lemma into high dimension (see e.g \cite{FY2022}). However, their results do not seem to adjust into our situation. Instead, we  consider the projection of the Cantor set onto the one-dimensional coordinate-axis. Newhouse thickness for any compact sets can be defined easily (See Section \ref{sec:Newhouse}). We have the following definition. 

\begin{definition}
Let $E$ be a Cantor set on ${\mathbb R}^d$. We say that $E$ is {\bf Newhouse projectively thick} if for all invertible linear transformations $T$,  the orthogonal projection of  $T(E)$ onto the $x_1$-axis has positive Newhouse thickness.
\end{definition}

 We now have the following theorem. 

\begin{theorem}\label{theorem_high_question}
    Let $E$ be a Cantor set on ${\mathbb R}^d$ that is Newhouse projectively thick. Then $E$ is not generically universal. 
\end{theorem}

This theorem covers many examples of Cantor sets. We will show that every  self-similar set on ${\mathbb R}^d$, not lying on a hyperplane, whose linear parts are rotation-free will be Newhouse projectively thick. We note that there has been intensive research about the dimensional properties of projections of Cantor sets (for a survey, see e.g. \cite{MR3558147}), but the properties of  Newhouse thickness along orthogonal projections that we present here appears to be new. We conjecture that all self-similar or self-affine sets, not lying on a hyperplane, are Newhouse projectively thick. 

\medskip

    \subsection{Some discussion and open problems.} Let us summarize our results and other known results in the  following table.

\medskip

\begin{center}
\small{{\begin{tabular}{|c|c|c|c|}
  \hline
  % after \\: \hline or \cline{col1-col2} \cline{col3-col4} ...
    &  Measure universal   &  Topologically universal  &  Generically universal\\
\hline
Finite sets & Yes  & Yes & Yes \\
\hline
Countably infinite sets& Unknown   & Yes & Yes\\
  \hline
Cantor sets on ${\mathbb R}^1$ & Unknown & No - Theorem \ref{theorem_invisible} &   No$^{*}$ - Theorem  \ref{theorem_positive_NW}\\
\hline
Cantor sets on ${\mathbb R}^d$, $d>1$ & Unknown & No - Theorem \ref{theorem_invisible} & No$^{*}$ - Theorem \ref{theorem_high_question} \\
\hline
Sets with interior& No & No & No\\
  \hline
  
\end{tabular}}}
\end{center}

\medskip

In the table, No$^{*}$ indicates a partial result established in this paper.  Theorem \ref{theorem_positive_NW} and Theorem \ref{theorem_high_question} refer to Cantor sets with positive Newhouse thickness on $\R^1$ and Newhouse projectively thick Cantor sets on $\R^d$ are not generically universal. It provides evidence that Cantor sets are unlikely to be generically universal. We believe that the following may be true, which draws an analogue of the Erd\H{o}s similarity conjecture  for generic sets.

\begin{conjecture}\label{conjecture_Cantor}
There are no generically universal Cantor sets on ${\mathbb R}^d$. 
\end{conjecture}  

\medskip

It is also reasonable that the following conjecture draws a parallel analogy of the Erd\H{o}s similarity conjecture in the purely topological non-measure-theoretic setting.

\begin{conjecture}\label{conjecture_uncountable}
There are no uncountable topologically universal sets on ${\mathbb R}^d$. 
\end{conjecture}

Unfortunately, Theorem \ref{theorem_invisible} does not imply the validity of Conjecture \ref{conjecture_uncountable}. This is because in the realm of descriptive set theory, it is known that with the axiom of choice, one can construct a so-called {\it Bernstein set} \cite[p.48]{Kechris}, in which neither the set nor its complement  contain a perfect set. i.e. the set contains no perfect subset and is uncountable. This means that we cannot use Theorem \ref{theorem_invisible} to conclude Bernstein sets is topologically non-universal. It is unclear to us whether Conjecture \ref{conjecture_uncountable} is even decidable within the ZFC axioms of set theory.   Nonetheless, despite such a pathological example, every uncountable Borel set (or more generally analytic set) of ${\mathbb R}^d$ contains a perfect subset (see \cite[p.85, 88]{Kechris}) so they will not  be topologically universal.

\medskip

The paper is organized as follows. We prove Theorem \ref{theorem_invisible} in Section \ref{sec:topol}. We will define Newhouse thickness for compact sets of ${\mathbb R}^1$ in Section \ref{sec:Newhouse}. We will prove our theorems on ${\mathbb R}^1$ in Section \ref{sec:generic-1} and then theorems on ${\mathbb R}^d$ in Section \ref{sec:generic-d}.

\section{Topological non-universality of Cantor sets} \label{sec:topol}

A function $h$ is called a {\bf dimension function/ gauge function} if $h: [0,1]\to [0,\infty)$ is  non-decreasing, continuous  and $h(0) = 0$. The $h-$Hausdorff measure is the translation -invariant Borel measure such that 
$$
{\mathcal H}^{h}(E) = \lim_{\delta\to 0} \inf\left\{\sum_{i=1}^{\infty} h (|U_i|): E\subset \bigcup_{i=1}^{\infty} U_i, \ |U_i|\le \delta\right\}
$$
where $|U|$ denotes the diameter of $U$. If $h(x) = x^s$, then ${\mathcal H}^h$ is the standard $s$-dimensional Hausdorff measure. 
\medskip

\begin{proposition}\label{prop_h}
For any dimension function $h$, there always exist  a dense $G_{\delta}$ set  $G$ on ${\mathbb R}^d$ such that ${\mathcal H}^{h}(G) = 0.$
\end{proposition}

\begin{proof}
For  any dimension function $h$, $h^{-1}$ may not exist since $h$ may not be strictly increasing. However, we define
$$
W(s) = \inf\{t>0: h(t)>s\} = \sup\{t>0: h(t)\le s\}. 
$$
Then we have $h(W(s))  = s$. Moreover, $W$ is strictly increasing whence $W(s)>0$ for all $s>0$. Let us now enumerate the rationals  ${\mathbb Q}^d =  \{r_1,r_2,...\}. $ Consider the following dense $G_{\delta}$ set
$$
G = \bigcap_{k=1}^{\infty} \left(\bigcup_{i=1}^{\infty} \left(r_i-\frac{W(2^{-(i+k)})}{2\sqrt{d}}, r_i+\frac{W(2^{-(i+k)})}{2\sqrt{d}}\right)^d\right).
$$
Then the diameter of the open squares inside the union is $W(2^{-(i+k)})$, so
$$
{\mathcal H}^h(G)\le \sum_{i=1}^{\infty} h (W(2^{-(i+k)})) = \sum_{i=1}^{\infty} 2^{-(i+k)} = 2^{-k}.
$$
As  $k$ is arbitrary, ${\mathcal H}^h(G) = 0$. The proof is complete. 
\end{proof}

\medskip

%Our main theorem is  as follows.

%\begin{theorem}\label{theorem_invisible}
%There is no topologically universal Cantor set on ${\mathbb R}^d$.
%\end{theorem}

We need a result from Rogers \cite[p.67]{Rogers}.

\begin{proposition}\label{prop_r}
Let $\Omega$ be an uncountable complete separable metric space. Then there exists a compact perfect set $C$ and a dimension function $h$ such that 
$$
0<{\mathcal H}^h(C)<\infty
$$
Consequently, suppose that a compact set $K$ in a metric space satisfies ${\mathcal H}^h (K) = 0$ for all dimension functions $h$. Then $K$ is a countable set. 
\end{proposition}

We remark that in \cite{Rogers}, dimension functions were defined to be right-continuous, but if we inspect the proof on page 65 in the book carefully, it is clear that we can construct $h$ to be continuous for the first statement.  For the second statement,  we note  that if $K$ is  uncountable and compact, then $K$ contains a perfect subset  $\Omega$. Applying the first statement, we have a perfect set $C\subset \Omega\subset K$ with ${\mathcal H}^h(C)>0$ for some dimension function $h$ which leads to a contradiction. 

\medskip

We should also remark on the other hand that there exists a Cantor set $K$ on ${\mathbb R}^1$ such that for all dimension functions $h$, either ${\mathcal H}^h(K) = 0$ or $\infty$ \cite{D71}. See also \cite{CDM13} for a recent survey. Nonetheless, in this case, we can still extract a sub-Cantor set of $K$ with finite positive Hausdorff measure for some gauge function $h$ by Proposition \ref{prop_r}. It means that ${\mathcal H}^h(K) =\infty$.

\medskip

Heuristically, to prove Theorem \ref{theorem_invisible}, we just take a suitable gauge function and a dense $G_{\delta}$ set according to Proposition \ref{prop_h}. Then the monotonicity of measure immediately leads to a contradiction. However, for general dimension functions,  we do not have a dilation formula for all invertible linear transformations.  Therefore,  we need the following lemma. 

\medskip

Recall that for any invertible linear transformation $T$, $\|T\|$ denotes the operator norm of $T:{\mathbb R}^d\to {\mathbb R}^d$ with $\R^d$ endowed with the Euclidean norm. i.e. 
$$
\|Tx\|\le \|T\|\|x\|
$$
holds for all $x\in\R^d$.

\begin{lemma}\label{lemma_Cantor_gauge}

Let $E \subset \mathbb{R}^{d}$ be a Borel set, $h$ a dimension function, and let $c>0$. Then the dimension function $h_c = h(cx)$ satisfies
$$
{\mathcal H}^{h_c}(T(E))\ge  {\mathcal H}^{h}(E) 
$$ 
for all $T$ such that $\|T^{-1}\| \le  c$.  

\end{lemma}

\begin{proof} First, from a direct observation we see that $h_c(x) = h(cx)$ is a dimension function. Let  $T$ such that $\|T^{-1}\| \le c$.  We note that  any covering $\bigcup_{i=1}^{\infty}V_i$ of $T(E)$ implies  that $\bigcup_{i=1}^{\infty}T^{-1}(V_i)$ is  a covering of $E$, so
$$
{\mathcal H}^h(E) \le \sum_{i=1}^{\infty} h (|T^{-1}V_i|).
$$
But from the definition of $\|T^{-1}\|$, the diameters satsify
$$
|T^{-1}V_i|\le \|T^{-1}\| |V_i|\le c |V_i|.
$$
Hence, 
$$
{\mathcal H}^{h}(E)\le  \sum_{i=1}^{\infty} h(|T^{-1}V_i|) \le  \sum_{i=1}^{\infty} h(c|V_i|) = \sum_{i=1}^{\infty} h_c(|V_i|).
$$
We now take infimum among all covers and obtain our desired conclusion.
\end{proof}

\medskip

\noindent{\it Proof of Theorem \ref{theorem_invisible}.} Let $E$ be a Cantor set on $\R^d$. By Proposition \ref{prop_r},  we can find a dimension function  $h$  such that ${\mathcal H}^h (E)>0$.  For each $n\in{\mathbb N}$, let us take the dimension function $h_n$ in Lemma \ref{lemma_Cantor_gauge} such that ${\mathcal H}^{h_n}(T(E))\ge {\mathcal H}^{h}(E)>0$ whenever $\|T^{-1}\|\le n$. 

\medskip

Now using Proposition \ref{prop_h}, we can find a dense $G_{\delta}$ set $G_n$ such that ${\mathcal H}^{h_n}(G_n) = 0$. By the Baire Category theorem, $G = \bigcap_{n=1}^{\infty}G_n$ is a dense $G_{\delta}$ set. We now claim that this $G$ cannot contain any affine copy of the Cantor set $E$. Indeed, suppose $t+T(E)$ is contained in $G$. Let $n\in{\mathbb N}$ be such that $\|T^{-1}\|\le n$, then $t+T(E)\subset G_n$. By taking ${\mathcal H}^{h_n}$ Hausdorff measure,  we find a contradiction  since ${\mathcal H}^{h_n}(G_n) = 0$, but $${\mathcal H}^{h_n}(t+T(E)) = {\mathcal H}^{h_n}(T(E))>0$$ by Lemma \ref{lemma_Cantor_gauge}.
\qquad$\Box$

\medskip

\section{Preliminaries on Newhouse thickness} \label{sec:Newhouse}
 The proof of our theorems on generic universality relies on the Newhouse gap lemma. The purpose of this section is to define the thickness and state the gap lemma that are necessary for our proof.  The definition of thickness and the gap lemma we use were first introduced by Newhouse \cite{Newhouse}. Our definition below is taken from the book of Palis and Takens  \cite{palis&takens}. We first need to define the gaps and bridges of Cantor sets in order to define Newhouse thickness.

\begin{definition}[Gap]
    Let $K$ be a Cantor set on ${\mathbb R}^1$.  A {\bf gap} of $K$ is a connected component of $\R \setminus K$.     A {\bf bounded gap} is a bounded connected component of $\R \setminus K$.      
\end{definition}
We now define the bridge of $C$ of Cantor set $K$. $|I|$ denotes the length of the interval $I$.  
\begin{definition}[Bridge, c.f. \cite{palis&takens}]
    Let $K$ be a Cantor set on ${\mathbb R}^1$ and $U= (u',u)$  be a bounded gap of $K$ with boundary point $u$.  The {\bf bridge} $C$ of $K$ at $u$ is the maximal interval on the right hand side of $u$  such that:
    \begin{itemize}
        \item $u$ is a boundary point of $C$
        \item $C$ contains no point of a gap $U'$ whose length $|U'| \geq |U|$.
    \end{itemize}
    We can define analogously the bridge for $u'$ by considering the maximal interval on the left hand side of $u'$ with the same property.  
\end{definition}

For clarity, Figure \ref{fig:thickness} shows that there may be smaller bounded gaps contained in $C$.  

\vspace*{0.25cm}
   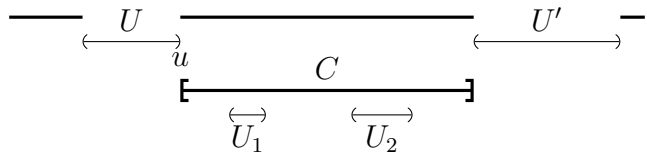
\begin{figure}[b]
        \centering
        \begin{tikzpicture}[scale=.65]
        \draw[very thick] (-6.5,0) -- (-5,0);
        \draw[very thick] (-3,0) -- (3,0); 
        \draw[very thick] (6,0) -- (6.5,0) ;
        %\draw[(-)] (-5,0)--(-3,0);
        \draw[{(-)}] (-5,-0.5) --  (-3,-0.5) node[midway, above]{$U$} node[below]{$u$};
        \draw[{(-)}] (3,-0.5) -- (6,-0.5) node[midway, above]{$U'$};
        \draw[(-)] (-2,-2)--(-1.25,-2)node[midway, below]{$U_1$};
        %\draw[very thick] (0.92,0) -- (1.92,0);
        \draw[(-)] (0.5,-2) -- (1.75,-2)node[midway, below]{$U_2$};
        
        \draw[very thick,below,{[-]}]  (-3,-1.5) -- (3,-1.5) node[midway, above]{$C$};
        \end{tikzpicture}
        \caption{At point $u$, we move to the right until we hit another gap of longer length. The interval travelled is the bridge $C$. Note that the Bridge contains gaps of smaller length than $U$ such as $U_1$ and $U_2.$ in the figure.}
        \label{fig:thickness}
        \end{figure}

We use this notion to define the Newhouse Thickness.  Intuitively the thickness of a Cantor set can be thought of as the infimum of ratios between the bounded gaps and the bridges.  
\begin{definition}[Newhouse Thickness for Cantor sets \cite{palis&takens}]\label{NWthick}
    The {\bf Newhouse Thickness} of $K$ at $u$ is defined as
    $$\tau(K, u) = \frac{|C|}{|U|}.$$
    Moreover, let  $\mathcal{U}$ be the set of all boundary points of bounded gaps in the Cantor set, the thickness of the entire Cantor set is 
    $$\tau(K) = \inf_{u\in \mathcal{U}} \tau(K, u) $$
\end{definition}

We will consider projections of Cantor sets in higher dimension onto the $x_1$-axis.  Such projections may not be perfect or may contain intervals, so we need to define the Newhouse thickness for general compact sets of ${\mathbb R}^1$. 

\medskip

We first recall some terminologies in point set topology \cite{BabyRudin}. Let $K\subset {\mathbb R}^1$ be a compact set; $x\in K$ is called a {\bf condensation point} of $K$ if every open neighborhood of $x$ contains uncountably many points of $K$. It is known that the set of all condensation points of $K$ is a perfect set inside $K$.  We call the set of all condensation points of $K$ the {\bf perfect part} of $K$.

\begin{definition}[Newhouse Thickness for general compact sets] Let $K$ be a compact set on ${\mathbb R}^1$ and let $P_K$ be the perfect part of $K$.   We now define
$$
\tau (K) = \left\{\begin{array}{ll} 0 &  \mbox{if $P_K = \emptyset$}\\\infty & \mbox{if $P_K$ contains an interval} \\  \tau(P_K) & \mbox{otherwise} \\\end{array}\right.
$$
\end{definition} 

\medskip

\begin{example}\label{example_digit}[Newhouse Thickness of the $N$-digit expansion Cantor Set]
    {\rm Let $N\ge 2$ and let $j\in\{1,...,N-2\}$. Define  $K$ to be the self-similar Cantor set by dividing $[0,1]$ into $N$ intervals of equal length, deleting the interval $[\frac{j}{N}, \frac{j+1}{N}]$ and repeating the process.   Then it is well-known that $K$ consists of all real numbers whose $N$-adic expansion omit the digit $j$:
    $$
    K = \left\{\sum_{k=1}^{\infty}\frac{d_k}{N^k}: d_k\in \{0,1,...,N-1\}\setminus\{j\}\right\}.
    $$
    Now, each gap at the $n$-th iteration is of length $N^{-n}$.  The Newhouse thickness is equal to $\min \{j, N-j-1\} $}.
    \end{example}

\medskip

We notice an important fact that Newhouse thickness is invariant under any invertible affine transformation, $x\mapsto t+\lambda x$ where $\lambda\ne 0$,  on ${\mathbb R}^1$.  The following lemma is now commonly referred to as the Newhouse Gap Lemma. 

\begin{lemma}  \label{gaplemma} {\bf (Newhouse Gap Lemma)}
    Let $K_1, K_2, \subset \R$ be Cantor sets with Newhouse thickness $\tau_1$ and  $\tau_2$ respectively and  $\tau_1 \cdot \tau_2 \ge 1$. Suppose that   $K_1$ is not contained in one of the gaps of $K_2$ and $K_2$ is not contained in one of  the gaps of $K_1$. Then $K_1 \cap K_2 \neq \emptyset.$
\end{lemma}
For additional information about the intersection in the above gap lemma, one can refer to  \cite{Astels_transaction}. We are now ready to prove our main results.

\medskip

\section{Generic non-universality of Cantor sets on ${\mathbb R}^1$.} \label{sec:generic-1}

We first prove our main theorems on ${\mathbb R}^1$. The construction of the $F_{\sigma}$ set in Equation (\ref{eqX}) in the proof below was motivated from \cite{MR2268116}, in which the authors constructed wavelets on a real line analogue of Cantor sets. The set in Equation (\ref{eqX}) is exactly the set they used. 

\medskip

\noindent{\bf Proof of Theorem \ref{theorem_positive_NW}.} We will establish the following claim:

\medskip
{\bf Claim: } Given an $\epsilon_{0} > 0$, there exists a dense $G_{\delta}$ set $G$ with $m(\R\setminus G) =0$ such that for any Cantor set $J$ with Newhouse thickness $\tau(J) \geq \epsilon_{0}$, $G$ contains no affine copy of $J$.

%For all $\epsilon_0>0$, we can find a dense $G_{\delta}$ set such that it does not contain affine copies of Cantor sets with Newhouse thickness at least $\epsilon_0>0$. 

\medskip

Assuming the claim, we construct a dense $G_{\delta}$ set $G_n$ of $m(\R\setminus G_n) =0$ with the property that it does not contain affine copies of Cantor sets with Newhouse thickness at least $1/n$. Then we consider 
$$
G = \bigcap_{n=1}^{\infty}G_n.
$$
Baire Category theorem ensures $G$ is a dense $G_{\delta}$ set. This $G$ will not contain any affine copy of any Cantor sets with positive Newhouse thickness. Moreover, by the subadditivity of measure, it is easy to see that $m(\R\setminus G) = 0.$ This will complete the proof.

\medskip

We now justify the claim. Let $\epsilon_0>0$ be given.  Consider the Cantor sets $K$ defined by contraction ratio $1/N$ and digits $\{0,1,...,N-1\}\setminus\{(N-1)/2\}$ and $N$ is odd as in Example \ref{example_digit}, we know that $\tau (K) = \frac{N-1}{2}$. Therefore,  we can find a  sufficiently large $N$ so that $\tau(K)>\epsilon_0^{-1}$. 

\medskip

Using the Cantor set $K$, we now define $X$ such that 
\begin{equation}\label{eqX}
X = \bigcup_{n \in \Z} \bigcup_{\ell \in \Z} N^n(K+\ell),    
\end{equation}
creating an $F_\sigma$ set. Now consider $X^c$.  Because $K^c$ is open and dense and so is its translated and dilated copies, $G = X^c$ is a dense $G_{\delta}$ and $m (\R\setminus G) = m(X) = 0$ as the Cantor set $K$ we constructed is of Lebesgue measure zero.  We now show that for any Cantor set $J$ with $\tau(J)\ge \epsilon_0$, $G = X^c$ contains no affine copy of $J$.

\medskip

 Suppose that we have some Cantor set $J$ with Newhouse thickness $\tau(J) \ge \epsilon_0$. Without loss of generality, by rescaling and translation,  we  can assume that the convex hull of $J$ is equal to $[0,1]$. We now fix any affine copy  $t+ \lambda J$ where $t\in\R$ and $\lambda\ne 0$. There exists a unique $n$ such that 
\begin{equation}\label{eq_lambda_bound}
    |\lambda| \in (N^{n-1}, N^n].
\end{equation}
Similarly there exists a unique $\ell$ such that 
\begin{equation}
t \in (\ell  N^n, (\ell+1)N^n].    
\end{equation}
Let 
$$K_1 = N^n(K+\ell) \text{ and } K_2 = t+ \lambda J.$$
The convex hull of $K_1,$ is  $[\ell  N^n, (\ell+1)N^n]$.  So, by our choice of $t$, we know that $K_2$ is not in the unbounded gap of $K_1$ and vice versa.

\medskip
Now we will check the construction of our Cantor sets such that each is not contained in the bounded gaps of the other. For $i=1,2$, we define $O_i$ to be the largest open bounded gap in $K_i$ and $I_i$ be the convex hull of $K_i$.  For $K_1$, we have  $|O_1| = N^{n-1}$ and  $ |I_1|=N^n$. For $K_2$, we recall that the convex hull of $J$ is $[0,1]$. Therefore, we have 
$$|O_2| =|\lambda|\cdot |O_J| \le |\lambda| \mbox{ and } \  |I_2| = |\lambda|$$ where $O_J$ is the largest  open bounded gap interval in $J$.  Therefore by our construction in (\ref{eq_lambda_bound}), the following two inequalities hold: $$|O_1|\leq |I_2| \text { and } |O_2| \leq |I_1|.$$ The inequalities imply that  $K_1$ is not fully contained in the bounded gaps of $K_2$ and $K_2$ is not fully contained in the bounded gaps of $K_1$.

\medskip

Since Newhouse thickness is invariant under affine transformation on ${\mathbb R}^1$, by our choice of $K$ we have that 
$$\tau(K_1)\tau(K_2) = \tau (K)\tau(J) \ge \epsilon_0^{-1}\cdot \epsilon_0 = 1.$$
 Therefore, the Gap Lemma in Lemma \ref{gaplemma} implies that $K_1 \cap K_2$ is non-empty and hence $K_2 = t+\lambda J$ intersects with one of the unions in $X$ in (\ref{eqX}). It implies that $t+\lambda J$ cannot be fully contained in the $G_{\delta}$ set $G = X^{c}$.  This establishes the claim, and therefore we conclude that $J$ is not topologically universal.
\qquad$\Box$

\medskip

%\noindent{\bf Proof of Corollary \ref{theorem_Svetic_question}.} Let $E$ be a Cantor set of positive Newhouse thickness. Consider the $G_{\delta}$ set $G$ we constructed in the Theorem \ref{theorem_positive_NW} and intersect it with the interval $[0,1]$.  By construction, the intersection does not contain any affine copies of $E$ and $G\cap [0,1]$ has full Lebesgue measure. This completes the proof.  
%Note that the complement of the $G_\delta$ set defined in (\ref{eqX}) is a countable union of scaled copies of Cantor sets of Lebesgue measure zero. Hence, it has Lebesgue measure zero, by the countable subadditivity of Lebesgue measure. As $G = \bigcap_{n=1}^{\infty} G_n$ with each $G_n$ is of the form in (\ref{eqX}), $G$ also has full Lebesgue measure.  Therefore the dense $G_{\delta}$ set  has full measure and it did not contain any affine copies of $E$.
%\qquad$\Box$

\medskip

\begin{remark}
{\rm We would like to remark that Bourgain proved that a Minkowski sum of three infinite sets cannot be measure universal. We can use this result to deduce that some Cantor sets of zero Newhouse thickness cannot be measure universal. Let $N_j\ge 2$ be integers and ${\mathcal D}_j\subset \{0,1,...,N_j-1\}$ be subsets of cardinality at least 2.  Define}
\begin{equation}\label{eqC}
C = \left\{\sum_{j=1}^{\infty} \frac{d_j}{N_1...N_j}: d_j\in{\mathcal D}_j\right\}.
\end{equation}
{\rm Then $C$ is not measure universal. Indeed, for $k= 0,1,2$, let }
$$
S_k = \left\{ \sum_{j\equiv k (\mbox{\rm mod} \ 3)} \frac{d_j}{N_1...N_j}: d_j\in{\mathcal D}_j \right\}.
$$
{\rm By the result of Bourgain, $C = S_0+S_1+S_2$ is a sum of three infinite sets and hence is not measure universal. Moreover, if $N_j\to\infty$, then the Cantor set $C$ above has zero Newhouse thickness. }

\medskip

{\rm On the other hand, our Theorem \ref{theorem_positive_NW}is independent from Bourgain's result in the sense that our construction of the avoiding set is explicit and of full Lebesgue measure, while the set constructed by Bourgain was not explicit and the Lebesgue measure is not known. Therefore, we still cannot determine if all above Cantor sets are generically universal if we merely use Bourgain's result. } %Furthermore, Ww should notice that not all Cantor sets are of the form in (\ref{eqC}).  the topological consideration also allowed us to construct one avoiding set for all Cantor sets of positive Newhouse thickness.}   
\end{remark}

\section{Generic non-universality of Cantor sets on ${\mathbb R}^d$.} \label{sec:generic-d}

We now turn to our results in higher dimensions. Our  first goal is  to show that some obvious examples cannot be generically universal. They include a set with  a path-connected component and embedding a lower dimensional generically non-universal set into higher dimensions.  

\medskip

\begin{proposition}\label{prop_path-connected}
    If $X\subset \R^d$ contains a path connected component, then $X$ is not generically universal.
\end{proposition}
\begin{proof}
    Let us consider the dense $G_{\delta}$ set that removes all the hyperplanes parallel that correspond to the coordinate hyperplanes shifted by rationals:
    $$
    G =  \bigcap_{i=1}^d \bigcap_{r \in \Q} {\mathbb R}^d \setminus\left\{(x_1,...,x_d)\in{\mathbb R}^d:x_i = r \right\}.
    $$
    This is clearly a dense $G_{\delta}$ set and $m (\R^d\setminus G)=0$ since there are only countably many hyperplanes and hyperplanes have $d$-dimensional Lebesgue measure zero. 
    Consider any affine copy of $X$. Then this affine copy must contain a path $L$.   The projection of $L$ onto the coordinate axes will be non-degenerate on some interval for at least one of the axes.  Call this the $i$-th axis.  This interval will contain a rational number $r$.  Therefore $L$ will intersect with the coordinate plane, $x_i = r$.
    In other words this dense $G_{\delta}$ cannot contain $L$. Thus, $X$ cannot be topologically universal.   
\end{proof}

The following simple lemma is needed in the following proofs. 

\begin{lemma}\label{lemma_product}
Let $G_1$ and $G_2$ be two dense $G_{\delta}$ sets in ${\mathbb R}^{d_1}$  and ${\mathbb R}^{d_2}$ respectively. Then $G_1\times G_2$ is a dense $G_{\delta}$ set in ${\mathbb R}^{d_1+d_2}$.   
\end{lemma}

\begin{proof}
Suppose that we write $G_1 = \bigcap_{n=1}^{\infty} O_n$ and $G_2 = \bigcap_{n=1}^{\infty} O_n'$ where $O_n$ and $O_n'$ are open dense sets in ${\mathbb R}^{d_1}$ and ${\mathbb R}^{d_2}$ respectively. The lemma follows immediately by observing that 
$$
G_1\times G_2 = \bigcap_{n=1}^{\infty} \bigcap_{m=1}^{\infty} O_n\times O_m'.
$$
\end{proof}

\medskip

\begin{proposition}\label{prop_embedding}
   Let $0<k<d$ be two positive integers. Suppose that  $E\subset {\mathbb R}^k$ is generically non-universal in ${\mathbb R}^k$. Then $E\times \{0\}$ cannot be generically universal in ${\mathbb R}^d$  ($0$ here is the $d-k$ dimensional zero vector).
\end{proposition}

\begin{proof}
Let ${\bf e}_i$ be the canonical coordinate basis in ${\mathbb R}^d$ and let $W ={\mathbb R}^k\times \{0\}.$ By our assumption, we can find a dense $G_{\delta}$ set $G_0\subset{\mathbb R}^k$ such that it does not contain $k$-dimensional affine copies of $E$. Let $G_0'$ be any dense $G_{\delta}$ set in ${\mathbb R}^{d-k}$ with $m ({\mathbb R}^{d-k}\setminus G_0') = 0$.  Then $G_0\times G_0'$ is a dense $G_{\delta}$ in ${\mathbb R}^d$.
By Fubini's theorem, $m\left((\R^k\setminus G_0)\times \R^{d-k}\right) = 0$, so is the other union.  We let $\Pi_{d,k}$ be the collection of all $k$-dimensional coordinate planes in ${\mathbb R}^d$. There are $\binom{d}{k}$ such planes. For each $P\in\Pi_{d,k}$, there exists a permutation matrix $\sigma_P$ such that 
$$
P = \sigma_P(W)
$$
We now define 
$$
G = \bigcap_{P\in\Pi_{d,k}} \sigma_P (G_0\times G_0').
$$
Note that 
$$
\R^d\setminus(G_0\times G_0') =\left( (\R^k\setminus G_0)\times \R^{d-k}\right)\cup\left( \R^k\times (\R^{d-k}\setminus G_0')\right).
$$
By Fubini's theorem, $m\left( (\R^k\setminus G_0)\times \R^{d-k}\right) = 0$, so is the other set in the above union.  
As $\sigma_P$ has unit determinant, we obtain that  $m(\R^d\setminus G) = 0$.

\medskip

To finish the proof, our next step is to show that $G$ cannot contain any affine copies of $E\times\{0\}$. To see this, we argue by contradiction. Suppose that there exists an invertible linear transformation $T$ on ${\mathbb R}^d$ such that $t+ T(E)\subset G$. Then the subspace
$$
T(W) = \mbox{span}\{T{\bf e}_1,..., T{\bf e}_k\}. 
$$
is $k$-dimensional and $\{T{\bf e}_1,..., T{\bf e}_k\}$ forms a basis for $T(W)$. Putting $T$ in matrix representation under the canonical basis. The matrix
$$
A = \left(\begin{array}{ccc}
    | & \cdots & | \\
    T{\bf e}_1 & \cdots & T{\bf e}_k\\
    | & \cdots &|
\end{array}\right)
$$
is of column rank $k$. Hence, it has row rank $k$ as well. Therefore, there exists $k$-linearly independent row vectors. Let ${\mathcal I} = \{i_1,...,i_k\}$ be the position of the row vectors of $A$ for which they are linearly independent. Let $A_{{\mathcal I}}$ be the square matrix whose rows are exactly the rows of $A$ at positions in ${\mathcal I}$. Then $A_{\mathcal I}$ is invertible on ${\mathbb R}^k$. 
Moreover, if we consider the $k$-dimensional coordinate plane $P$ at those $x_{i_1},...,x_{i_k}$ axes and denote by $P_{\mathcal I}$ the orthogonal projection onto $P$, then we have
$$
P_{\mathcal I}(t+T(E)) = P_{\mathcal I}(t)+ A_{\mathcal I}(E)
$$ 
and 
$$
P_{\mathcal I} (\sigma_P (G_0\times G_0')) = G_0.
$$
By the construction of $G$, $t+T(E)\subset \sigma_P (G_0\times G_0')$, meaning that $P(t)+ A_{\mathcal I}(E)\subset G_0$. As $A_{\mathcal I}$ is invertible, we find an affine copy of $E$ inside $G_0$, which is a contradiction. This completes the proof. 

\end{proof}

As we know already that the middle-third Cantor set is not generically universal, the above proposition shows that it cannot be embedded to become generically universal in higher dimensions either. Notice also that such an embedding of a Cantor set will never be Newhouse projectively thick since the projection will always be a singleton in the orthogonal complement. We are now ready to prove our main theorem on ${\mathbb R}^d$ stated in the introduction. 

\medskip

\medskip

\noindent{\bf Proof of Theorem \ref{theorem_high_question}.}  Suppose we have a Newhouse projectively thick Cantor set $J$ on $\R^d$.  We now take $G_0$ in Theorem \ref{theorem_positive_NW} and construct 
$$
G = G_0\underbrace{\times \dots \times}_\text{$d$-times}G_0. 
$$
Applying Lemma \ref{lemma_product}, $G_0\times \dots \times G_0 $ is a dense $G_\delta$ set in ${\mathbb R}^d$ and therefore $G$ is also a dense $G_\delta$ set. With Fubini's theorem, it is not difficult to show that ${\mathbb R}^d\setminus G$ has zero Lebesgue measure. 
 
 \medskip

It remains to prove that $G$ has no affine copy of $J$.   Assume to the contrary that $G$ contains an affine copy of $J$ and denote it by $t+T(J)$. Then
$$
t+T(J) \subset G_0\underbrace{\times \dots \times}_\text{$d$-times} G_0.
$$
Denote by $P$ the orthogonal projection onto the $x_1$-axis. We have $P [t+T(J)] \subset G_0$. By linearity we can express the orthogonal projection $P [t+T(J)]  $ as $ P(t)+ P[T(J)]$. We have that $G_0$ contains an affine copy of $P[T(J)]$. But $J$ is Newhouse projectively thick which implies that  $\tau (P[T(J)])>0$.  We obtain a contradiction since, by Theorem \ref{theorem_positive_NW}, $G_{0}$ cannot contain any affine copies of ${ P}[T(J)]$.  This completes the proof. \qquad$\Box$

\medskip

To conclude this paper, we consider a class of self-similar sets that are Newhouse projectively thick. Recall that if we are given finitely many contractive similarity maps $\phi_i: {\mathbb R}^d\to {\mathbb R}^d$, $i=1,...,N$, such that 
$$
\phi_i(x) = \rho_i O_i x+b_i
$$
where $0<\rho_i<1$, $O_i$ is an orthogonal transformation and $b_i\in{\mathbb R}^d$, $\Phi = \{\phi_i: i=1,...,N\}$ forms an {\bf iterated function system (IFS)} and there exists a unique non-empty compact set $K = K_{\Phi}$ such that 
$$
K = \bigcup_{i=1}^{N}\phi_i (K).
$$
We say that the IFS is {\bf rotation-free} if all $O_i$ are identity transformations. We also say that a self-similar set is {\bf non-degenerate} if it is not contained in any hyperplane of ${\mathbb R}^d$

\begin{example}
All non-degenerate self-similar sets on ${\mathbb R}^d$ generated by rotation-free IFS must be Newhouse projectively thick. 
\end{example}
\begin{proof}
Let $P$ be the orthogonal projection onto the $x_1$-axis and let $T$ be any invertible linear transformation. We note that for a rotation free IFS, the set $PT(K)$ is still  generated by a self-similar IFS on ${\mathbb R}^1$ with maps
$$
\widetilde{\phi}_i(x) = \rho_i x+ PT(b_i). 
$$
Notice that the self-similar set is non-degenerate, meaning that $PT(K)$ is not a singleton. The self-similar set $PT(K)$ is a compact perfect set. 
In Feng and Wu \cite[Lemma 3.5]{FengWu2022}, the authors showed that all self-similar sets not lying on a hyperplane have a positive thickness $\tau_{FW}$ defined in \cite[Definition 1.1]{FengWu2022}.  On the other hand,  it was claimed without proof in the paragraph after Definition 1.1 in \cite{FengWu2022} that  if $d = 1$, then $\tau_{FW}(E)>0$ if and only if the Newhouse thickness $\tau(E)>0$. This would have implied that $\tau(PT(K))>0$. 

\medskip

For the self-containment of this paper, we justify the direction required in this proof in the following claim:

\medskip
 
{\bf Claim:} If $d=1$, then $\tau_{FW}(E)>0$ implies that $\tau(E)>0$. 

\medskip

To see this claim, Definition 1.1 in \cite{FengWu2022} states that 
$$
\tau_{FW} (E) = \sup \left\{c\ge 0: \forall x\in E, \forall r\in(0,|E|], \exists \ y\in \R \ \mbox{s.t.} \  \mbox{conv}(B(x,r)\cap E)\supset B(y,cr)\right\}.
$$
Here, $|E|$ is the diameter of $E$, conv$(K)$ means the convex hull of a set $K$, and $B(x,r)$ denotes the Euclidean ball centered at $x$ of radius $r$. For each fixed $x\in E$ and $r\in (0,|E|]$, we define 
$$
\tau_{FW} (E, x, r) = \sup \left\{c\ge 0:  \exists \ y\in \R \ \mbox{s.t.} \  \mbox{conv}(B(x,r)\cap E)\supset B(y,cr)\right\}.
$$
Then $\tau_{FW} (E) = \inf\limits_{x\in E}\inf\limits_{r\in(0,|E|]}\tau_{FW} (E,x, r)$.

Suppose that $\tau_{FW}(E)>0$. Consider $u\in{\mathcal U}$ (using the notation as in Definition \ref{NWthick}) where $u$ is a boundary point of a bounded gap $U$. Consider the open interval $(u-|U|, u+|U|)$. Then one of the endpoints of $\mbox{conv}(B(u,|U|)\cap E)$ is $u$.  Let $C$ be the bridge of $u$ and without loss of generality assume $C$ is on the right hand  side of $U$.   

If $u+|U|\in C$, then  $|C|\ge |U|$ and $\tau (E, u)\ge 1$.  If, however, $u+|U|\not\in C$, then because of the definition of the bridge, $u+|U|$ is in the first gap whose length is larger than $|U|$. Hence,   $\mbox{conv}(B(u,|U|)\cap E) = C = B(z, \frac{|C|}{2|U|}|U|)$ for some center $z$. This means that $\tau_{FW} (E, u, |U|) = \frac{|C|}{2|U|}$.  Then $$\tau (E,u) = \frac{|C|}{|U|}= 2 \cdot \tau_{FW} (E, u, |U|)\ge 2 \cdot \tau_{FW}(E)>0.$$ Taking infimum among all $u\in {\mathcal U}$, we show that $\tau(E)\ge \min\{2 \tau_{FW}(E),1\}>0$. This completes the proof of the claim. 

%If $y\in E$ also, then   $\mbox{conv}(B(u,|U|)\cap E) = [u, y]$ and $y = u+|U|$. We have  $|C|\ge |U|$ and   $\tau (E,u)\ge 1$.  On the other hand, if $y\not\in E$. Then $y$ is in one of the smaller gaps inside $C$, say $U'$. As there are infinitely many such gaps, $|U'|< |U|$ 

Coming back to the proof, we now know that all self-similar sets, not a singleton, on ${\mathbb R}^1$ must have a positive Newhouse thickness. So the self-similar set $PT(K)$ has a positive Newhouse thickness. This shows that all non-degenerate self-similar sets generated by rotation-free IFS must be Newhouse projectively thick. 
\end{proof}

\end{document}